\documentclass{article}

\usepackage{amsmath,amsfonts,amsthm,amssymb,amscd,color,xcolor,mathrsfs,verbatim,microtype}
\usepackage{graphicx,eurosym}
\usepackage{hyperref}

\usepackage[applemac]{inputenc}

\usepackage[cyr]{aeguill}

\colorlet{darkblue}{blue!50!black}

\hypersetup{
    colorlinks,%
    citecolor=blue,%
    filecolor=red,%
    linkcolor=darkblue,%
    urlcolor=blue,%
    pdfnewwindow=true,%
    pdfstartview={FitH}
}

\usepackage{graphicx,amscd,mathrsfs,wrapfig,mathrsfs,lipsum}
\usepackage{eufrak}
\usepackage{float}
\usepackage{tikz}
\usepackage{multicol}
\usepackage{caption}
\usetikzlibrary{arrows}
\usepackage{capt-of}

\colorlet{darkblue}{blue!50!black}

\renewcommand{\Im}{\mathop{\rm Im}\nolimits}

\newcommand{\p}{\partial}
\newcommand{\e}{\varepsilon}

\newcommand{\Q}{{\mathbb Q}}

\newcommand{\R}{{\mathbb R}}
\newcommand{\Z}{{\mathbb Z}}
\newcommand{\IP}{{\mathbb P}}
\newcommand{\pP}{{\mathbb P}}
\newcommand{\I}{{\mathbb I}}
\newcommand{\E}{{\mathbb E}}
\newcommand{\T}{{\mathbb T}}

\newcommand{\ty}{\infty}

\newcommand{\XXX}{{\boldsymbol{X}}}

\newcommand{\Ker}{\mathop{\rm Ker}\nolimits}
\newcommand{\Ind}{\mathop{\rm Ind}\nolimits}
\newcommand{\codim}{\mathop{\rm codim}\nolimits}

\newcommand{\OOmega}{{\boldsymbol\Omega}}

\newcommand{\oomega}{{\boldsymbol\omega}}

\newcommand{\BB}{{\cal B}}

\newcommand{\DD}{{\cal D}}

\newcommand{\FF}{{\cal F}}
\newcommand{\GG}{{\cal G}}
\newcommand{\HH}{{\cal H}}
\newcommand{\II}{{\cal I}}

\newcommand{\KK}{{\cal K}}

\newcommand{\PP}{{\cal P}}

\newcommand{\VV}{{\cal V}}

\newcommand{\lag}{\langle}
\newcommand{\rag}{\rangle}

\newcommand{\dd}{{\textup d}}

\newcommand{\PPPP}{{\mathfrak P}}

\newcommand{\lspan}{\mathop{\rm span}\nolimits}

\newcommand{\supp}{\mathop{\rm supp}\nolimits}

\theoremstyle{plain}
\newtheorem*{mt}{Main Theorem}

\newtheorem*{lemma*}{Lemma}
\newtheorem{theorem}{Theorem}[section]
\newtheorem{lemma}[theorem]{Lemma}
\newtheorem{proposition}[theorem]{Proposition}
\newtheorem{corollary}[theorem]{Corollary}
\theoremstyle{definition}
\newtheorem{definition}[theorem]{Definition}

\theoremstyle{remark}

\newtheorem{example}[theorem]{Example}

\numberwithin{equation}{section}

\begin{document}
\author{Sergei Kuksin\footnote{Institut de Math\'emathiques de Jussieu--Paris Rive Gauche, CNRS, Universit\'e Paris Diderot, UMR 7586, Sorbonne Paris Cit\'e, F-75013, Paris, France \& School of Mathematics, Shandong University, Jinan, PRC \& Saint Petersburg State University, Universitetskaya nab., St. Petersburg, Russia; e-mail: \href{mailto:Sergei.Kuksin@imj-prg.fr}{Sergei.Kuksin@imj-prg.fr}} \and Vahagn~Nersesyan\footnote{Laboratoire de Math\'ematiques, UMR CNRS 8100, UVSQ, Universit\'e Paris-Saclay, 45, av. des Etats-Unis, F-78035 Versailles, France \& Centre de Recherches Math\'ematiques, CNRS UMI 3457, Universit\'e de Montr\'eal, Montr\'eal,  QC, H3C 3J7, Canada;  e-mail: \href{mailto:Vahagn.Nersesyan@math.uvsq.fr}{Vahagn.Nersesyan@math.uvsq.fr}} \and
Armen Shirikyan\footnote{Department of Mathematics, University of Cergy-Pontoise, CNRS UMR 8088, 2 avenue Adolphe Chauvin, 95302 Cergy--Pontoise, France \& Department of Mathematics and Statistics,
McGill University, 805 Sherbrooke Street West, Montreal, QC, H3A 2K6, Canada;  
e-mail: \href{mailto:Armen.Shirikyan@u-cergy.fr}{Armen.Shirikyan@u-cergy.fr}}}
\title{Mixing via controllability for randomly forced nonlinear dissipative PDEs}
\date{}
\maketitle

\begin{abstract}
In the paper~\cite{KNS-2018}, we studied the problem of mixing for a class of PDEs with very degenerate noise and established the uniqueness of stationary measure and its exponential stability in the dual-Lipschitz metric. One of the hypotheses imposed on the problem in question required that the unperturbed equation should have exactly one globally stable equilibrium point. In this paper, we relax that condition, assuming only global controllability to a given point. It is proved that the uniqueness of a stationary measure and convergence to it are still valid, whereas the rate of convergence is not necessarily exponential. The result is applicable to randomly forced parabolic-type PDEs, provided that the deterministic part of the external force is in general position, ensuring a regular structure for the attractor of the unperturbed problem. 


\end{abstract}

\newpage
\tableofcontents

\setcounter{section}{-1}

\section{Introduction}
\label{s0} 

In the last twenty years, there was a substantial progress in the question of description of the long-time behaviour of solutions for PDEs with random forcing. The problem is particularly well understood when all the determining modes are directly affected by the stochastic perturbation. In this situation, for a large class of PDEs the resulting random flow possesses a unique stationary distribution, which attracts the laws of all the solutions with an exponential rate. We refer the reader to~\cite{FM-1995,KS-cmp2000,EMS-2001,BKL-2002} for the first results in this direction and to the review papers~\cite{ES-2000,bricmont-2002,debussche-2013} and the book~\cite{KS-book} for a detailed discussion of the literature. The question of uniqueness of stationary distribution becomes much more delicate when the random forcing is very degenerate and does not act directly on all the determining modes of the evolution. In this case, the propagation of the randomness under the unperturbed dynamics plays a crucial role and may still ensure the uniqueness and stability of a stationary distribution. There are essentially two mechanisms of propagation---transport and diffusion---and they allowed one to get two groups of results. The first one deals with  random forces that are localised in the Fourier space. In this situation, it was proved by Hairer and Mattingly~\cite{HM-2006,HM-2011} that the Navier--Stokes flow is exponentially mixing in the dual-Lipschitz metric, provided that the random perturbation is white in time. F\"oldes, Glatt-Holtz, Richards, and Thomann~\cite{FGRT-2015} established a similar result for the Boussinesq system, assuming that a degenerate random force acts only on the equation for the temperature. The recent paper~\cite{KNS-2018} deals with various parabolic-type PDEs perturbed by {\it bounded\/} observable forces, which allowed for treatment of nonlinearities of arbitrary degree. The second group of results concerns random forces localised in the physical space. They were obtained in~\cite{shirikyan-asens2015,shirikyan-2018} for the Navier--Stokes equations in an arbitrary domain with a random perturbation distributed either in a subdomain or on the boundary. 

The goal of the present paper is to relax a hypothesis in~\cite{KNS-2018} that required the existence of an equilibrium point which is globally asymptotically stable under the unperturbed dynamics. To illustrate our general result, let us consider the following example of a randomly forced parabolic PDE to which it is applicable:
\begin{equation} \label{rf-pde}
	\p_t u-\nu\Delta u+f(u)=h(x)+\eta(t,x), \quad x\in\T^d,  \quad d\le 4. 
\end{equation}
Here $\nu>0$ is a parameter, $f:\R\to\R$ is a polynomial   satisfying some natural growth and dissipativity hypotheses (see~\eqref{E:4.1} and~\eqref{E:4.2}), $h: \T^d\to \R$ is a smooth deterministic function, and~$\eta$ is a finite-dimensional {\it Haar coloured noise}. More precisely, we assume that~$\eta$ is a random process that takes  values in a sufficiently large\footnote{More precisely, we require~$\HH$ to be {\it saturating\/} in the sense of Definition~\ref{D:4.1}.} finite-dimensional  subspace~$\HH$ of~$L^2(\T^d)$ and has  the form 
\begin{equation} \label{eta-intro}
	\eta(t,x)=  \sum_{i\in\II}b_i \eta^i(t)\varphi_i(x),
\end{equation}
where $\{\varphi_i\}_{i\in\II}$ is an orthonormal basis in~$\HH$, $\{b_i\}$ are non-zero numbers,   and~$\{\eta^i\}$ are   independent copies of a random process   defined by
\begin{equation} \label{0.3}
	\tilde\eta(t)=\sum_{k=0}^\infty \xi_k h_0(t-k)
	+\sum_{j=1}^\infty c_j\sum_{l=0}^\infty \xi_{jl}h_{jl}(t).
\end{equation}
In this sum,  $\{h_0,h_{jl}\}$ is the Haar basis\footnote{Note that the Haar basis used in this work differs from that of~\cite[Section~22]{lamperti1996} by normalisation.} in $L^2(0,1)$ (see~\cite[Section~5.2]{KNS-2018}), $\{c_j\}$ is a sequence given by
\begin{equation}\label{cj}
	c_j=C j^{-q}\quad \text{for some } C>0, \,\, q>1,
\end{equation} 
 and~$\{\xi_k,\xi_{jl}\}$  are independent identically distributed (i.i.d.) scalar random variables with Lipschitz-continuous density~$\rho$ such that $ \supp\rho\subset[-1,1]$   and $\rho(0)>0$.  Let us supplement Eq.~\eqref{rf-pde} with the initial condition
 \begin{equation}\label{rf-pde0}
 u(0,x)=u_0(x),
 \end{equation}
 where $u_0\in L^2(\T^d)$. Under the above hypotheses, the restrictions to integer times of solutions for problem~\eqref{rf-pde}, \eqref{rf-pde0} form a discrete-time Markov process, which is denoted by~$(u_k,\IP_u)$, and this Markov process is the main subject of our study.

We assume that the space~$\HH$ and the functions~$f$ and~$h$ are in general position in the sense that the following   two conditions are satisfied.
\begin{description}
\item[\hypertarget{S}{(S) Stationary states.}]
{\sl The nonlinear elliptic equation
\begin{equation} \label{stationary-eq}
	-\nu\Delta w+f(w)=h(x), \quad x\in\T^d
\end{equation}
has finitely many solutions $w_1,\dots,w_N\in H^2(\T^d)$\/}.
\end{description}  
Genericity of this condition is proved in~Section~\ref{S:5.3}, and    examples are provided by the criterion established in \cite[Section~5]{CI-1974}; e.g., in our context with $d=1$, one can take $f(u)=u^3-u$ and $h=0$. 

The existence of a Lyapunov function (see~\eqref{E:4.5}) implies that at least one of the stationary states, say~$w_N$, is 
  {\it locally asymptotically stable}.\,\footnote{To see this, it suffices to note that the Lyapunov function admits at least one local minimum, and any local minimum is a locally asymptotically stable stationary~state.}  This means that, for some  number $\delta>0$, the solutions  of the unperturbed equation  
\begin{equation} \label{unperturbed-eq}
	\p_t u-\nu\Delta u+f(u)=h(x)	
\end{equation}
that are issued from an initial condition~$u_0$ with  $\|u_0-w_N\|_{L^2(\T^d)} \le \delta$ converge uniformly to~$w_N$:
\begin{equation} \label{wNconvergence}
\lim_{t\to +\ty}\sup_{u_0\in B(w_N,\delta)}\|u(t)-w_N\|_{L^2(\T^d)}=0,	
\end{equation}
where $B(w,\delta)$ is the ball in~$L^2$ of radius~$\delta$ centred at~$w$. To formulate the second condition, let us   denote by $\KK$ the support of the law for  the restriction to the interval~$[0,1]$ of the process~\eqref{eta-intro} and by $S_n(u_0;\zeta_1,\dots,\zeta_n)$ the value of the solution for problem~\eqref{rf-pde}, \eqref{rf-pde0} in which the external force~$\eta$ coincides with~$\zeta_k$ on the time interval $[k-1,k]$. 
 
\begin{description}
\item[\hypertarget{C}{(C) Controllability to the neighbourhood of $w_N$.}]
{\sl  For any $1\le i\le N-1$,  there is an integer $n_i$ and functions~$\zeta_{i1},\dots,\zeta_{in_i}\in \KK$ such that \begin{equation}\label{controlc}
\|S_{n_i}(w_i;\zeta_{i1},\dots,\zeta_{in_i})-w_N\|_{L^2(\T^d)} < \delta.
\end{equation}} 
\end{description}
The validity of this condition       can be derived from Agrachev--Sarychev type approximate controllability results,\footnote{Theorem~\ref{T:5.5} of the Appendix establishes an approximate controllability property for Eq.~\eqref{rf-pde}. Namely, it shows that, for any $i\in[\![1,N-1]\!]$, there is an $\HH$-valued function~$\zeta_i$ such that the trajectory of Eq.~\eqref{rf-pde} issued from~$w_i$ is in the open $\delta$-neighbourhood of~$w_N$ at time $t=1$. Replacing the process~$\eta$ in~\eqref{eta-intro} with~$a\eta$ and choosing  $a\ge1$ sufficiently large, we can ensure that $\KK^a:=\supp \DD(a\eta)$ contains a function arbitrarily close to~$\zeta_i$, so that inequality~\eqref{controlc} holds with $n_i=1$.} provided that the support $\KK$ is  sufficiently large. The following theorem is a consequence of the main result of this paper on the uniqueness and mixing of a stationary measure for~$(u_k,\IP_u)$. Its exact formulation and further discussions are presented in Section~\ref{s4}. 

\begin{mt}
Under the above conditions,   the Markov process~$(u_k,\IP_u)$ has a unique stationary measure~$\mu$ on the space~$L^2(\T^d)$, and for any other solution~$u(t)$ of~\eqref{rf-pde}, we have
	$$
	\|\DD(u(k))-\mu\|_L^*\to0\quad\mbox{as $k\to\infty$},
	$$
	where $\|\cdot\|_L^*$ stands for the dual-Lipschitz metric over the space~$L^2(\T^d)$, and~$\DD(\cdot)$ denotes the law of a random variable. 
\end{mt}

The paper is organised as follows. In Section~\ref{s1}, we formulate and discuss our main theorem on the uniqueness of a stationary measure and mixing for a discrete-time Markov process. In Section~\ref{s2}, we derive some preliminary results needed in the proof of the main theorem, which is established in Section~\ref{s3}. Application to a class of nonlinear parabolic PDEs is presented in Section~\ref{s4}. Finally, the Appendix gathers some auxiliary results. 

\subsubsection*{Acknowledgement}
This research was supported by the {\it Agence Nationale de la Recherche\/} through  the grants ANR-10-BLAN~0102 and ANR-17-CE40-0006-02. SK  thanks the {\it Russian Science Foundation\/} for support through the grant  18-11-00032. VN and AS were supported by the CNRS PICS {\it Fluctuation theorems in stochastic systems\/}. The research of AS was carried out within the MME-DII Center of Excellence (ANR-11-LABX-0023-01) and supported by {\it Initiative d'excellence Paris-Seine\/}. The authors thank R.~Joly for the proof of the genericity of Hypothesis~\hyperlink{S}{\rm(S)} (see Proposition~\ref{P:5.3}). 

\subsubsection*{Notation}
For a Polish space~$X$ with a metric $d_X(u,v)$, we denote by $B_X(a,R)$ the closed ball of radius $R > 0$ centred at $a\in X$ and by $\dot B_X(a,R)$ the corresponding open ball. The Borel $\sigma$-algebra on~$X$ and the set of probability measures are denoted by~$\BB(X)$ and~$\PP(X)$, respectively. We shall use the following spaces, norms, and metrics. 

\smallskip
\noindent
$C_b(X)$ denotes the space of bounded continuous functions $f:X\to\R$ endowed with the norm $\|f\|_\infty=\sup_X|f|$, and $L_b(X)$ stands for the space of functions $f\in C_b(X)$ such that
$$
\|f\|_L:=\|f\|_\infty+\sup_{0<d_X(u,v)\le 1}\frac{|f(u)-f(v)|}{d_X(u,v)}<\infty.
$$
In the case of a compact space~$X$, we write~$C(X)$ and $L(X)$.

\smallskip
\noindent
The space~$\PP(X)$ is endowed with either the {\it total variation metric\/} or the  {\it dual-Lipschitz\/}  metric. They are defined by 
\begin{align} 
	\|\mu_1-\mu_2\|_{\mathrm{var}}&:=\sup_{\Gamma\in\BB(X)}|\mu_1(\Gamma)-\mu_2(\Gamma)|
=\frac12\sup_{\|f\|_\infty\le1}
\left|\lag f,\mu_1\rag-\lag f,\mu_2\rag\right|, \label{TVN}\\
\|\mu_1-\mu_2\|_L^*
&:=\sup_{\|f\|_L\le1}\left|\lag f,\mu_1\rag-\lag f,\mu_2\rag\right|, \label{dL}
\end{align}
where $\mu_1, \mu_2\in \PP(X)$, and $\lag f,\mu\rag = \int_X f(u) \,\mu(\dd u)$ for $f\in C_b(X)$ and $\mu\in \PP(X)$.

\smallskip
\noindent
$L^p(J,E)$ is the space Borel-measurable functions~$f$ on an interval $J\subset\R$ with range in a Banach space~$E$ such that 
$$
\|f\|_{L^p(J,E)}=\biggl(\int_J \|f(t)\|_E^p\dd t\biggr)^{1/p}<\infty;
$$
in the case $p=\infty$, this norm should be modified accordingly.

\smallskip
\noindent
$H^s(D)$ denote the Sobolev space of order $s\ge0$ with the usual norm~$\|\cdot\|_s$. 

\section{Main result}
\label{s1} 

Let us denote by~$H$ and~$E$ separable Hilbert spaces and by $S:H\times E\to H$ a continuous mapping. Given a sequence $\{\eta_k\}$ of i.i.d.\ random variables in~$E$, we consider the random dynamical system (RDS) 
\begin{equation} \label{1.1}
u_k=S(u_{k-1},\eta_k), \quad k\ge1. 
\end{equation}
In what follows, we always assume that the law~$\ell$ of the random variables~$\eta_k$ has a compact support~$\KK\subset E$ and that there is a compact set $X\subset H$ such that $S(X\times\KK)\subset X$. Our aim is to study the long-time behaviour of the restriction of the RDS~\eqref{1.1} to the invariant set~$X$. 

For a vector $u\in H$ and a sequence $\{\zeta_k\}\subset E$, we set $S_m(u;\zeta_1,\dots,\zeta_m):=u_m$, where $\{u_k\}$ is defined recursively by Eq.~\eqref{1.1} in which $u_0=u$ and $\eta_k=\zeta_k$. We assume that the hypotheses below hold for the RDS~\eqref{1.1} and some Hilbert space~$V$ compactly embedded into~$H$.

\smallskip
\begin{description}
\item[\hypertarget{H1}{(H$_1$)} Regularity.] 
{\sl The mapping~$S$ is twice continuously differentiable from $H\times E$ to~$V$, and its derivatives are bounded on bounded subsets. Moreover, for any fixed $u\in H$, the mapping $\eta\mapsto S(u,\eta)$ is analytic from~$E$ to~$H$, and all its derivatives $(D_\eta^jS)(u,\eta)$ are continuous functions of~$(u,\eta)$ that are bounded on bounded subsets of~$H\times E$.}

\item[\hypertarget{H2}{(H$_2$)} Approximate controllability to a point.] 
{\sl There is $\hat u\in X$ such that, for any $\e>0$, one can find an integer $m\ge1$ with the following property: for any $u\in X$ there are $\zeta_1,\dots,\zeta_m\in\KK$ such that}
\begin{equation} \label{1.2}
\|S_m(u;\zeta_1,\dots,\zeta_m)-\hat u\|<\e.
\end{equation}
\end{description}

Given $u\in X$, let us denote by~$\KK^u$ the set of those $\eta\in E$ for which the image of the derivative $(D_\eta S)(u,\eta)$ is dense in~$H$. It is easy to see that~$\KK^u$ is a Borel subset in~$E$; see Section~1.1 in~\cite{KNS-2018}. 

\begin{description}
\item[\hypertarget{H3}{(H$_3$)} Approximate controllability of the linearisation.] 
{\sl The set~$\KK^u$ has full $\ell$-measure  for any $u\in X$.}

\item[\hypertarget{H4}{(H$_4$)} Structure of the noise.]
{\sl There exists an orthonormal basis~$\{e_j\}$ in~$E$, independent random variables~$\xi_{jk}$, and real numbers~$b_j$ such that 
\begin{equation} \label{1.3}
\eta_k=\sum_{j=1}^\infty b_j\xi_{jk}e_j,\quad B:=\sum_{j=1}^\infty b_j^2<\infty. 
\end{equation}
Moreover, the laws of~$\xi_{jk}$ have Lipschitz-continuous densities~$\rho_j$ with respect to the Lebesgue measure on~$\R$.}
\end{description}

We refer the reader to Section~1.1 in~\cite{KNS-2018} for a discussion of these conditions and of their relevance in the study of large-time asymptotics of trajectories for PDEs with random forcing. Here we only mention that the approximate controllability hypothesis~\hyperlink{(H2)}{(H$_2$)} imposed in this paper is weaker than the dissipativity condition of~\cite{KNS-2018} and allows one to treat a much larger class of PDEs that possess several steady states. A drawback is that the main result of this paper does not give any estimate for the rate of convergence (to the unique stationary measure), which remains an interesting open problem.

\smallskip
To formulate our main abstract result, we introduce some notation. Since $\{\eta_k\}$ are i.i.d.\ random variables, the trajectories of~\eqref{1.1} issued from~$X$ form a discrete-time Markov process, which is denoted by~$(u_k,\IP_u)$. We shall write $P_k(u,\Gamma)$ for its transition function and~$\PPPP_k:C_b(X)\to C_b(X)$ and $\PPPP_k^*:\PP(X)\to\PP(X)$ for the corresponding Markov operators. 

Let us recall that a measure $\mu \in \PP(H)$ is  said to be {\it stationary\/} for $(u_k,\IP_u)$ if~$\PPPP^*_1\mu=\mu$. The continuity of~$S$ implies that $(u_k,\IP_u)$ possesses the Feller property, and by the Bogolyubov--Krylov argument and the compactness of~$X$, there is at least one stationary measure. We wish to investigate its uniqueness and stability.

Let ~$\|\cdot\|_L^*$ be the dual-Lipschitz metric on the space of probability measures on~$X$ (see Notation). The following theorem, which is the main result of this paper, describes the behaviour of~$\PPPP_k^*$ as the time goes to infinity. 

\begin{theorem} \label{t1.1}
Suppose that Hypotheses \hyperlink{H1}{\rm(H$_1$)}--\hyperlink{H4}{\rm(H$_4$)} are satisfied. Then the Markov process $(u_k,\IP_u)$ has a unique stationary measure $\mu\in \PP(X)$, and there is a sequence of positive numbers~$\{\gamma_k\}$ going to zero as $k\to\infty$ such that
\begin{equation} \label{1.6}
\|\PPPP_k^*\lambda-\mu\|_L^*\le \gamma_k\quad\mbox{for all $k\ge0$ and $\lambda\in\PP(X)$}. 
\end{equation}
\end{theorem}
A proof of this result is given in Section~\ref{s3}. Here we discuss very  briefly the main idea, postponing the details to Section~\ref{s3.1}.

A sufficient condition for the validity of the conclusions is given by  Theorem~\ref{t5.1} in the Appendix. According to that result, it suffices to check the recurrence and stability properties. The recurrence is a simple consequence of the approximate controllability to the point~$\hat u$; see Hypothesis~\hyperlink{H2}{(H$_2$)}. The proof of stability is much more involved and will follow from two properties, \hyperlink{(A)}{(A)} and~\hyperlink{(B)}{(B)}, of Theorem~\ref{T:1.2}. Their verification is based on a key new idea of this work, which reduces the required properties to a study of a conditional random walk. The latter is discussed in Section~\ref{s2}, together with  an auxiliary result on the transformation of the noise space~$E$ (which was established in~\cite{KNS-2018}).

\section{Preliminary results}
\label{s2} 

\subsection{Transformation in the control space}
Given a number $\delta>0$, we set $D_\delta:=\{(u,u')\in X\times H:\|u-u'\|\le\delta\}$. The following proposition is established in Section~3.2 of~\cite{KNS-2018} (see Proposition~3.3 with $\sigma=1/4$). 

\begin{proposition}\label{P:2.5} 
Suppose that Hypotheses~\hyperlink{H1}{\rm(H$_1$)}, \hyperlink{H3}{\rm(H$_3$)}, and~\hyperlink{H4}{\rm(H$_4$)} are satisfied. Then, for any $\theta\in(0,1)$, there are positive numbers $C$, $\beta$, and~$\delta$, a family of Borel subsets $\{\KK^{u,\theta}\subset \KK^u\}_{u\in X}$, and a measurable   mapping $\varPhi : X\times H \times E  \to E$ such that $\varPhi^{u,u'}(\eta)=0$ if $\eta\notin \KK^{u,\theta}$ or $u'=u$, and
\begin{align}
\ell(\KK^{u,\theta})&\ge 3/4,
\label{3.7}\\
\|\ell-\varPsi_*^{u,u'}( \ell)\|_{\mathrm{var}}&\le C\, \|u-u'\|^\beta,
 \label{3.8}\\
 \|S(u,\eta)- S(u',\varPsi^{u,u'}(\eta))\| &\le \theta\, \|u-u'\|,
\label{3.9}
\end{align} 
 where   $\varPsi^{u,u'}(\eta):=\eta+\varPhi^{u,u'}(\eta)$, $\varPsi_*^{u,u'}( \ell)$ is   the image of the measure $\ell$ under~$\varPsi^{u,u'}$,  and  $(u,u')\in D_\delta$ and $\eta\in \KK^{u,\theta}$ are arbitrary points. 
\end{proposition}

\subsection{Asymptotic properties of a conditional random walk}
\label{s2.1}
Given a real-valued random variable~$\xi$ defined on a probability space $(\Omega,\FF,\IP)$ and a sub-$\sigma$-algebra $\GG\subset\FF$, we denote by~$\mu_\xi(\omega,\dd x)$ the {\it conditional law of~$\xi$ given~$\GG$\/}. In other words, $\mu_\xi$ is a random probability measure on~$\R$ with the underlying space~$(\Omega,\GG)$ such that 
\begin{equation*}
\E\bigl(f(\xi)\,|\,\GG)=\int_\R f(x)\mu_\xi(\omega,\dd x),
\end{equation*}
where $f:\R\to\R$ is any Borel-measurable function such that $f(\xi)\in L^1(\Omega,\IP)$. 
In what follows, we shall write $\IP(\xi\in\Gamma\,|\,\GG)$ for $\mu_\xi(\omega,\Gamma)$. 

\smallskip
We now fix a number $p\in(\frac12,1)$ and filtration $\{\FF_k\}_{k\ge0}$ and consider a sequence of random variables $\{w_k\}_{k\ge1}$ such that $w_k$ is $\FF_k$-measurable and 
\begin{equation} \label{2.1}
\IP\{w_k=1\,|\,\FF_{k-1}\}=p, \quad \IP\{w_k=-1\,|\,\FF_{k-1}\}=1-p. 
\end{equation}
Let us define 
\begin{equation}\label{2.2}
\zeta_k=\sum_{j=1}^kw_j, \quad M_k=\zeta_k-(2p-1)k,
\end{equation}
with the convention $\zeta_0=M_0=0$. 

\begin{proposition} \label{p2.1}
\begin{itemize}
\item[\bf(a)]
The family $\{\zeta_k^{(m)}=m+\zeta_k\}_{m\in\Z}$ is a discrete-time Markov process with the phase space~$\Z$. 
\item[\bf(b)]
For any $\e>0$ there is a random time~$\tau=\tau(\e,p)\ge1$ and a number $\alpha=\alpha(\e,p)>0$ such that 
\begin{gather}
M_k\ge -\e k\quad\mbox{for $k\ge\tau$},\label{2.4}\\
\E \,e^{\alpha\tau}<\infty. \label{2.5}
\end{gather}
\item[\bf(c)]
For any integer $l\ge0$, we have
\begin{equation}\label{2.3}
\IP\bigl\{\zeta_k> -l\mbox{ for all }k\ge0\bigr\}=1-\bigl(\tfrac{1-p}{p}\bigr)^l.
\end{equation}

\end{itemize}
\end{proposition}

\begin{proof}
{\bf(a)}\ 
We fix a bounded function $f:\Z\to\R$ and use~\eqref{2.1} to write
\begin{align*}
\E\bigl(f(\zeta_k^{(m)})\,|\,\FF_{k-1}\bigr) 
&=\E\bigl(f(\zeta_{k-1}^{(m)}+w_k)\,|\,\FF_{k-1}\bigr) \\
&=\E\bigl(f(\zeta_{k-1}^{(m)}+1)I_{\{w_k=1\}}
+f(\zeta_{k-1}^{(m)}-1)I_{\{w_k=-1\}})\,|\,\FF_{k-1}\bigr)\\
&=f(\zeta_{k-1}^{(m)}+1)\,\IP\{w_k=1\,|\,\FF_{k-1}\}\\
&\quad +f(\zeta_{k-1}^{(m)}-1)\,\IP\{w_k=-1\,|\,\FF_{k-1}\}\\
&=f(\zeta_{k-1}^{(m)}+1)p+f(\zeta_{k-1}^{(m)}-1)(1-p)
=\E f(\zeta_1^{(l)})\,\bigr|_{l=\zeta_{k-1}^{(m)}}. 
\end{align*}

\smallskip
{\bf(b)}\ 
Let us note that   $\widetilde w_j=w_j-(2p-1)$ are random variables whose absolute values are bounded by~$2$ and variances are equal to~$\sigma_p^2=4p(1-p)$. By inequality~(7) with $M=2$ in the proof of Lemma~1 of~\cite[Section~12]{lamperti1996}, we have 
\begin{equation} \label{4.19}
\E\,e^{-t\tilde w_j}\le \exp\bigl(2p(1-p)t^2+4t^3\bigr),
\end{equation}
where $0\le t\le 1$ is arbitrary. Combining~\eqref{4.19} with the Markov property, we derive
$$
\E\,e^{-t M_k} 
=\E\,\E\bigl(e^{-t M_k}\,|\,\FF_{k-1}\bigr)
\le \exp\bigl(2p(1-p)t^2+4t^3\bigr)\E \,e^{-tM_{k-1}}.
$$
Iterating this inequality, we obtain
\begin{equation} \label{4.20}
\E\,e^{-t M_k} \le e^{(2p(1-p)t^2+4t^3)k}, \quad k\ge1.
\end{equation}

We now fix $\e>0$, define the events $\Gamma_k=\{-M_k\ge \e k\}$, and use the Borel--Cantelli lemma. It follows from~\eqref{4.20} and the Chebyshev inequality
$$
\IP(\Gamma_k)\le e^{-t\e k}\,\E\,e^{-tM_k}
\le \exp\bigl(-tk\bigl[\e-2p(1-p)t-4t^2\bigr]\bigr). 
$$
Taking $t=\frac{\e}{4p(1-p)}$ and assuming that $\e\le p^2(1-p)^2$, we derive
\begin{equation} \label{4.21}
\IP(\Gamma_k)\le \exp\bigl(-\gamma k\bigr), \quad \gamma=\gamma(\e,p)=\frac{\e^2}{16p(1-p)}. 
\end{equation}
Since the series $\sum_k\IP(\Gamma_k)$ converges, the random variable 
$$
\sigma=\min\{n\ge1:M_k\ge -\e k\mbox{ for $k\ge n$}\}
$$
is almost surely finite. Moreover, in view of~\eqref{4.21}, for $0<\alpha<\gamma$, we have
\begin{align*}
\E\, e^{\alpha\sigma}&=\sum_{k=1}^\infty \IP\{\sigma=k\}e^{\alpha k} 
\le e^{\alpha}+\sum_{k=2}^\infty \IP(\Gamma_{k-1})e^{\alpha k}
\le e^{\alpha}+\sum_{k=2}^\infty e^{-\gamma k+\alpha k}<\infty.
\end{align*}
We thus obtain~\eqref{2.5} with $\alpha=\frac{\e^2}{32p(1-p)}$.

\smallskip 
{\bf(c)}\ 
Let us consider the hitting time 
$$
\tau_l=\min\{k\ge1: \zeta_k=l\},
$$
 with the convention that $\tau_l=\infty$ if~$\zeta_k$ does not reach~$l$. We need to prove that, for any~$l\ge0$, 
\begin{equation} \label{2.6}
\IP\{\tau_{-l}<\infty\}=\bigl(\tfrac{1-p}{p}\bigr)^l. 
\end{equation} 
To this end, given any integers $a\le m\le b$, we define
$$
P_m(a,b)=\IP_m\{\tau_a<\tau_b\},
$$
where the subscript~$m$ in the right-hand side indicates that the probability is calculated 
for~$\zeta_k^{(m)}$.   Using (b) with any $\e\in (0,2p-1)$, it is straightforward to see that 
$$
\IP\{ \tau_b<\ty\}=1 \quad\mbox{for any  $ b\ge  0$}.
$$
It follows that, up to sets of measure zero, for any $b\ge0$, we have
$$
 \{\tau_{-l}<\infty\}=\bigcup_{r=b}^\infty\{\tau_{-l}<\tau_r\}.
$$
Since $\{\tau_{-l}<\tau_r\}$ is an increasing sequence with respect to~$r$, we conclude that
\begin{equation} \label{2.7}
\IP\{\tau_{-l}<\infty\}=\lim_{b\to\infty} \IP\{\tau_{-l}<\tau_b\}= \lim_{b\to\infty}P_0(-l,b).
\end{equation} 
If we prove that 
\begin{equation} \label{2.8}
P_m(a,b)=\frac{\varkappa_p^m-\varkappa_p^b}{\varkappa_p^a-\varkappa_p^b}, 
\end{equation} 
where $\varkappa_p=\frac{1-p}{p}$, then the required equality~\eqref{2.6} will follow from~\eqref{2.7}.

To prove~\eqref{2.8}, we apply an argument\footnote{Note that our situation is slightly different, since the jumps~$w_j$ are not independent.} in~\cite[Section~XIV.2]{feller1968} (see the proof of~(2.8) there). Using the Markov property 
and the fact that $\zeta_1^{(m)}=m\pm1$ on the set $w_1=\pm1$, we write 
\begin{align*}
P_m(a,b)&=\E_m\IP_m\{\tau_a<\tau_b\,|\,\FF_1\}
=\E_m\Bigl(\IP_{\zeta_1^{(m)}}\{\tau_a<\tau_b\}\bigl(I_{\{w_1=1\}}+I_{\{w_1=-1\}}\bigr)\Bigr)\\
&=pP_{m+1}(a,b)+(1-p)P_{m-1}(a,b).
\end{align*}
We thus obtain a difference equation for the numbers $\{P_m(a,b),a\le m\le b\}$, which satisfy
the boundary conditions $P_a(a,b)=1$ and $P_b(a,b)=0$. 
A simple calculation shows that the only solution is given by~\eqref{2.8}.
\end{proof}

\begin{corollary} \label{c2.2}
For any $c\in(0,2p-1)$, there is a sequence $\{p_l\}\subset \R$   depending only on $c$ and $p$   such that 
\begin{gather} 
 \IP\{\zeta_k\ge -l+ck\mbox{ for all $k\ge0$}\}\ge p_l \quad \mbox{for all $l\ge1$},\label{2.12}\\
 p_l\to 1 \quad \mbox{as $l\to \ty$}. \label{2.12b} 
\end{gather}
\end{corollary}
\begin{proof} 
Applying~\eqref{2.4} with $\e=2p-1-c$, we see that $\zeta_k\ge ck$ for $k\ge \tau$.
By the Chebyshev inequality and \eqref{2.5}, we have
$$
\pP\{\tau>l\}\le C e^{-\alpha l} \quad\mbox{ for $l\ge1$}.
$$ 
 It follows that 
\begin{equation} \label{2.13}
\IP\{\zeta_k\ge ck\mbox{ for $k\ge l$}\}\ge 1-C e^{-\alpha l}\quad\mbox{ for $l\ge1$}. 
\end{equation}
On the other hand, it follows from~\eqref{2.3} that 
$$
\IP\{\zeta_k\ge -l+ck\mbox{ for $0\le k\le l$}\}\ge1- \bigl(\tfrac{1-p}{p}\bigr)^{[(1-c)l]}
\quad \mbox{for $l\ge1$}, 
$$
where  $[a]$ stands for the integer part of $a$.   Combining this with~\eqref{2.13}, we obtain~\eqref{2.12} with
$$
p_l:=1- \bigl(\tfrac{1-p}{p}\bigr)^{[(1-c)l]} -C e^{-\alpha l}.
$$
Since $c<1$, we have limit \eqref{2.12b}. 
\end{proof}

\subsection{Continuous probability measures}
\label{s5.2}
Let $(\Omega,\FF,\IP)$ be a probability space. We shall say that~$\IP$ is {\it continuous\/} if for any $\Gamma\in\FF$ and $p\in[0,\IP(\Gamma)]$ there is $\Gamma_p\in\FF$ such that $\Gamma_p\subset\Gamma$ and $\IP(\Gamma_p)=p$. Given a measurable space $(X,\BB)$ and measurable mapping $F:\Omega\to X$, we say that~{\it $\IP$ admits a disintegration with respect to~$\Q=F_*(\IP)$\/} if there is a random probability measure $\{P(x,\cdot)\}_{x\in X}$ on~$(\Omega,\FF)$ such that 
\begin{equation} \label{5.21}
\IP\bigl(A\cap F^{-1}(B)\bigr)=\int_B P(x,A)\Q(\dd x)\quad
\mbox{for any $A\in\FF$, $B\in\BB$}. 
\end{equation}
The following result provides a simple sufficient condition for continuity of a probability measure. 

\begin{lemma} \label{l5.2}
Let $(\Omega,\FF,\IP)$ be a probability space and let $F:\Omega\to\R$ be  a measurable mapping such that $\Q=F_*(\IP)$ has a density~$\rho$ with respect to the Lebesgue measure and~$\IP$ admits a disintegration~$P(s,A)$ with respect to~$\Q$. Then~$\IP$ is continuous. 
\end{lemma}

\begin{proof}
Given $\Gamma\in\FF$, we define $\Gamma(r)=\Gamma\cap F^{-1}((-\infty,r])\in\FF$, where $r\in\R$. Then $\IP(\Gamma(r))$ converges to~$0$ as $r\to-\infty$ and to~$\IP(\Gamma)$ as $r\to+\infty$. Moreover, by~\eqref{5.21}, we have
$$
\IP\bigl(\Gamma(r)\bigr)=\int_{-\infty}^r P(s,\Gamma)\rho(s)\,\dd s, 
$$
whence we see that the function $r\mapsto \IP(\Gamma(r))$ is continuous. The required result follows from the intermediate value theorem. 
\end{proof}

We now apply the above idea to deal with a construction that will be used in Section~\ref{s3}. Namely, let $(\Omega_i,\FF_i,\IP_i)$, $i=1,2$ be two probability spaces and let $(\Omega,\FF,\IP)$ be their direct product. With a slight abuse of notation, we write~$\FF_i$ for the sub-$\sigma$-algebra on~$\Omega$ generated by the natural projection $\Omega\to\Omega_i$. 

\begin{lemma} \label{l5.3}
In addition to the above hypotheses, suppose that the probability space $(\Omega_2,\FF_2,\IP_2)$ and a function $F:\Omega_2\to\R$ satisfy the conditions of Lemma~\ref{l5.2}, and let $\Gamma\in\FF$ be such that, for some $p\in(0,1)$, 
\begin{equation} \label{5.22}
\E(I_\Gamma|\,\FF_1)\ge p\quad\mbox{$\IP$-almost surely}. 
\end{equation}
Then there is $\Gamma'\in\FF$ such that $\Gamma'\subset\Gamma$ and 
\begin{equation} \label{5.23}
\E(I_{\Gamma'}|\,\FF_1)= p\quad\mbox{$\IP$-almost surely}. 
\end{equation}
\end{lemma}

\begin{proof}
We first reformulate the lemma in somewhat different terms. Given $\Gamma$ and~$\omega_1\in\FF_1$, we denote 
$$
\Gamma(\omega_1)
=\{\omega_2\in\Omega_2:(\omega_1,\omega_2)\in\Gamma\}. 
$$
It is straightforward to check that 
$$
\E(I_\Gamma|\,\FF_1)=\IP_2\bigl(\Gamma(\omega_1)\bigr). 
$$
Furthermore, the inclusion $\Gamma'\subset\Gamma$ holds if and only if $\Gamma'(\omega_1)\subset\Gamma(\omega_1)$ for any $\omega_1\in\Omega_1$. Thus, the lemma is equivalent to the following assertion: if $\Gamma\in\FF$ is such that $\IP_2(\Gamma(\omega_1))\ge p$ for $\IP_1$-a.e.\ $\omega_1\in\Omega_1$, then there is $\Gamma'\in\FF$  such that $\Gamma'(\omega_1)\subset\Gamma(\omega_1)$ for any $\omega_1\in\Omega_1$ and $\IP_2(\Gamma'(\omega_1))= p$ for $\IP_1$-a.e.\ $\omega_1\in\Omega_1$. 

Given a real-valued measurable function $r(\omega_1)$, we define 
$$
\Gamma'=\{(\omega_1,\omega_2)\in\Gamma:
F(\omega_2)\le r(\omega_1)\}\subset\Gamma. 
$$
Then $\Gamma'(\omega_1)=\Gamma(\omega_1)\cap F^{-1}((-\infty,r(\omega_1)])$, so that
\begin{equation} \label{5.24}
\E(I_{\Gamma'}|\,\FF_1)=\IP_2(\Gamma'(\omega_1))
=\int_{-\infty}^{r(\omega_1)}
P\bigl(s,\Gamma(\omega_1)\bigr)\rho(s)\,\dd s,
\end{equation}
where $P(s,\cdot)$ stands for the disintegration of~$\IP_2$ with respect to~$F_*(\IP_2)$, and~$\rho$ is the density of~$F_*(\IP_2)$ with respect to the Lebesgue measure. Consider the measurable function 
$$
G(\omega_1, t) = \int_{-\infty}^t P(s, \Gamma(\omega_1)) \rho(s)ds, \quad t\in\R.
$$ 
It is continuous in~$t$, vanishes when $t=-\infty$ and is $\ge p$ when $t=+\infty$.  Consider the set $\{t\in \R:G(\omega_1,t)\le p\}$. This is a measurable set which is the sub-graph of certain measurable function $t=\lambda(\omega_1)$.  Choosing $r(\omega_1) = \lambda(\omega_1)$, we see that the right-hand side of~\eqref{5.24} is identically equal to~$p$, and $\Gamma'$ is a pre-image of the above-mentioned sub-graph under  the measurable mapping $(\omega_1, \omega_2)\mapsto (\omega_1, F(\omega_2))$. We conclude that~$\Gamma'$ is measurable, which completes the proof. 
\end{proof}

\section{Proof of the main theorem}
\label{s3} 
\subsection{General scheme}
\label{s3.1}
We wish to apply a sufficient condition for mixing from~\cite[Section~3.1.2]{KS-book}, stated below as Theorem~\ref{t5.1}. To this end, we need to check the recurrence and stability conditions. The recurrence is a consequence of Hypothesis~\hyperlink{H1}{(H$_2$)}. Indeed, inequality~\eqref{1.2} implies that $P_m(u,B_X(\hat u,r))>0$ for any $u\in X$ and some integer $m=m_r\ge1$. Since the function $u\mapsto P_m(u,\dot B_X(\hat u,r))$ is lower semicontinuous and positive, it is separated from zero on the compact set~$X$, so that~\eqref{5.1} holds. We thus need to prove the stability.   We shall always assume that the hypotheses of Theorem~\ref{t1.1} are satisfied. Recall that, given $\delta>0$, we write $D_\delta=\{(u,u')\in X\times H:\|u-u'\|\le\delta\}$. The following result provides a sufficient condition for the validity of~\eqref{5.2}. 

\begin{theorem} \label{T:1.2}
Suppose there is  a measurable mapping $\varPsi:X\times H\times E\to E$, taking $(u,u',\eta)$ to~$\varPsi^{u,u'}(\eta)$, and positive numbers  $\alpha$, $\beta$, and $q\in(0,1)$ such that $\varPsi^{u,u}(\eta)=\eta$ for any $u\in X$ and $\eta\in E$, and the following properties hold.
\begin{description}
\item[\hypertarget{(A)}{(A)} Stabilisation.]
For any $u,u'\in H$, let $(u_k,v_k)$ be defined by 
\begin{align}
(u_0,v_0)&=(u,u'), \label{1.8}\\
(u_k,v_k)&=(S(u_{k-1},\eta_k),S(v_{k-1},\varPsi^{u_{k-1},v_{k-1}}(\eta_k)).
\label{1.9}
\end{align}
Let us introduce the stopping time 
\begin{equation} \label{tau}
\tau=\min\bigl\{k\ge1:\|u_k-v_k\|> q^{k}\|u-u'\|^\alpha\bigr\}
\end{equation}
and, for any $\delta>0$, define the quantity 
$$
p(\delta)=\inf_{(u,u')\in D_\delta}\IP\{\tau=+\infty\}. 
$$
Then  
\begin{equation} \label{1.10}
\lim_{\delta\to0} p(\delta)=1.
\end{equation}
\item[\hypertarget{(B)}{(B)} Transformation of measure.] 
For any $(u,u')\in X\times H$, we have   
\begin{equation} \label{1.12}
\|\ell-\varPsi_*^{u,u'}(\ell)\|_{\mathrm{var}}\le C \|u-u'\|^\beta.
\end{equation}
\end{description}
Then condition~\eqref{5.2} is valid: 
\begin{equation} \label{107}
\lim_{\delta\to0^+}\sup_{(u,u')\in D_\delta}\sup_{k\ge0}
\|P_k(u,\cdot)-P_k(u',\cdot)\|_L^*=0. 
\end{equation} 
	\end{theorem}

Theorem~\ref{T:1.2} is established in Section~\ref{s3.2}. Note that if the constant~$C$ in the right-hand side of inequality~\eqref{1.12} vanishes, then the random variables~$\eta_k$ and $\eta_k':=\varPsi^{u_{k-1},v_{k-1}}(\eta_k)$ form a coupling for the pair of measures~$(\ell,\ell)$, so that $\DD(v_k)=\DD(u_k')$, where $\{u_k',k\ge0\}$ solves~\eqref{1.1} with $u_0'=u'$. In this case, we deal with the classical coupling approach to compare $P_k(u,\cdot)$ and~$P_k(u',\cdot)$. Our proof of Theorem~\ref{t1.1} crucially uses the above result with $\|u-u'\|\ll1$. The right-hand side of~\eqref{1.12} is not zero in this situation, but it is small, so we deal with a kind of approximate coupling. 

To prove Theorem~\ref{t1.1} given Theorem~\ref{T:1.2}, it suffices to construct a measurable mapping~$\varPsi$ satisfying Conditions~\hyperlink{(A)}{(A)} and~\hyperlink{(B)}{(B)}. This will be done with the help of Proposition~\ref{P:2.5}. 

\subsection{Proof of Theorem~\ref{T:1.2}}
\label{s3.2}
Let us define a probability space $(\Omega,\FF,\IP)$ by the relations
$$
\OOmega=\{\oomega=(\omega_k)_{k\ge1}:\omega_k\in E\}, \quad 
\FF=\BB(\OOmega),\quad \IP=\bigotimes_{k=1}^\infty\ell,
$$
where $\OOmega$ is endowed with the  Tikhonov topology. 
Let $(u_k(\oomega), v_k(\oomega))$  be the trajectory of~\eqref{1.8}, \eqref{1.9} with  $\eta_k\equiv\omega_k$ and let~$u_k'(\oomega)$ be the trajectory of~\eqref{1.1} with~$u_0=u'$ and $\eta_k\equiv\omega_k$. Given $u,u',y,z\in H$, we define mappings $\theta_k:E\to E$, $k\ge1$ by
\begin{equation} \label{3.08}
\theta_k(y,z,\omega)
=\left\{
\begin{array}{cl}
\varPsi^{y,z}(\omega)
&\mbox{if} \quad \|y-z\|
\le q^{k-1}\|u-u'\|^\alpha,\\[3pt]
\omega &\mbox{if}\quad \|y-z\|
> q^{k-1}\|u-u'\|^\alpha,
\end{array}
\right.
\end{equation} 
where $\Psi$ is constructed in Proposition~\ref{P:2.5}, and consider the mapping 
$$
\Theta:\OOmega\to\OOmega,\quad 
\Theta(\oomega)
=\bigl(\theta_k(u_{k-1}(\oomega),v_{k-1}(\oomega),\omega_k)\bigr)_{k\ge1}.
$$ 
Clearly, $\{u_k(\oomega)\}_{k\ge0}$ is a trajectory of~\eqref{1.1} with $u_0=u$, and 
\begin{equation} \label{3.09}
v_k(\oomega)=u_k'(\Theta(\oomega))\quad
\mbox{for $k\ge1$, $\oomega\in\{\tau=+\infty\}$}.
\end{equation}
We now write 
\begin{equation} \label{3.10}
\|P_k(u,\cdot)-P_k(u',\cdot)\|_L^*
\le \|P_k(u,\cdot)-\DD(v_k)\|_L^*+\|\DD(v_k)-P_k(u',\cdot)\|_L^*
\end{equation}
and estimate the two terms on the right-hand side. For $(u,u')\in D_\delta$, we have
\begin{align}
\|P_k(u,\cdot)-\DD(v_k)\|_L^*
&=\sup_{\|F\|_L\le1}\bigl|\E(F(u_k)-F(v_k))\bigr|\notag\\
&\le2\IP\{\tau<\infty\}+\E\bigl(I_{\{\tau=\infty\}}\|u_k-v_k\|\bigr)\notag\\
&\le 2(1-p(\delta))+\delta^\alpha q^{k}.
\label{3.011}
\end{align}
To estimate the second term on the right-hand side of~\eqref{3.10}, we use the following simple result, in which $G=\{\tau=+\infty\}$ (e.g., see Section~7.2 in~\cite{KNS-2018} for a proof). 

\begin{lemma} \label{l5.4}
Let $(\Omega,\FF,\IP)$ be a probability space, let~$X$ be a Polish space, and let~$U,V:\Omega\to X$ two random variables. Suppose there is a measurable mapping  $\Theta:\Omega\to\Omega$ such that 
\begin{equation} \label{5.25}
U(\Theta(\omega))=V(\omega)\quad\mbox{for $\omega\in G$},
\end{equation}
where $G\in\FF$. Then
\begin{equation} \label{5.26}
\|\DD(U)-\DD(V)\|_{\mathrm{var}}
\le 2\,\IP(G^c)+\|\IP-\Theta_*(\IP)\|_{\mathrm{var}}. 
\end{equation}
\end{lemma}

In view of~\eqref{3.09} and~\eqref{5.26}, we have  
\begin{align}
\|\DD(v_k)-P_k(u',\cdot)\|_L^*
&\le 2\,\|\DD(v_k)-P_k(u',\cdot)\|_{\mathrm{var}}\notag\\
&\le 4\,\IP\{\tau<\infty\}+2\|\IP-\Theta_*(\IP)\|_{\mathrm{var}}.
\label{3.012}
\end{align}
The first term on the right-hand side does not exceed $4(1-p(\delta))$. Substituting~\eqref{3.012} and~\eqref{3.011} in~\eqref{3.10}, we see that~\eqref{107} will be established if we show~that 
\begin{equation} \label{1.15}
\sup_{(u,u')\in D_\delta}
\|\IP-\Theta_*(\IP)\|_{\mathrm{var}}\to0\quad\mbox{as $\delta\to0$}. 
\end{equation}

To prove this, we use the second relation in~\eqref{TVN} to calculate the total variation distance between two measures $\mu_1,\mu_2\in\PP(\OOmega)$.  Obviously, it suffices to consider the functions~$F$ belonging to a dense subset of $C(\OOmega)$ and satisfying the inequality $\|F\|_\infty\le 1$. Hence, the supremum can be taken over all functions depending on finitely many coordinates. 

We thus fix any integer $m\ge1$ and consider an arbitrary continuous  function $F:\OOmega\to\R$ of the form $F(\oomega)=F(\omega_1,\dots,\omega_m)$   with~$\|F\|_\infty\le 1$. Then
\begin{align}
\langle F,\IP-\Theta_*(\IP)\rangle
&=\E\left\{F(\omega_1,\dots,\omega_m)
-F(\theta_1(u,u',\omega_1),\dots,\theta_m(u_{m-1},v_{m-1},\omega_m)\right\}
\notag\\
&=\sum_{k=1}^m \E F_k(u,u',\omega_1,\dots,\omega_m), 
\label{1.16}
\end{align}
where we set
\begin{align*}
F_k(u,u',\omega_1,\dots,\omega_m)
&=F(\theta_1(u,u',\omega_1),\dots,\theta_{k-1}(u_{k-2},v_{k-2},\omega_{k-1}),\omega_k,\dots,\omega_m)\\
&-F(\theta_{1}(u,u',\omega_1),\dots,\theta_{k}(u_{k-1},v_{k-1},\omega_{k}),\omega_{k+1},\dots,\omega_m). 
\end{align*}
Let~$\FF_k\subset\FF$ be the $\sigma$-algebra generated by the first~$k$ coordinates. Setting
$$
\Delta_k=F(x_1,\dots,x_{k-1},\omega_k,\dots,\omega_m)
-F(x_1,\dots,x_{k-1},\theta_{k}(y,z,\omega_{k}),\omega_{k+1},\dots,\omega_m),
$$
we note that
$$
|\E\,\Delta_k|\le \|\ell-\theta_{k*}(y,z,\ell)\|_{\mathrm{var}}
\le I_{[0,q^{k-1}\|u-u'\|^\alpha]}(\|y-z\|)\|\ell-\varPsi_*^{y,z}(\ell)\|_{\mathrm{var}},
$$
where we used~\eqref{3.08}. Combining this with~\eqref{1.12}, we derive
\begin{align*}
\bigl|\E\bigl(F_k(u,u')\,|\,\FF_{k-1}\bigr)\bigr|=|\E\,\Delta_k|
\le Cq^{\beta(k-1)}\|u-u'\|^{\alpha\beta},
\end{align*}
where one takes $x_j =\theta_j(u_{j-1},v_{j-1},\omega_j)$, $y=u_{k-1}$, and $z=v_{k-1}$ in the middle term after calculating the mean value. 
Substituting this into~\eqref{1.16}, we obtain
$$
|\langle F,\IP\rangle-\langle F,\Theta_*(\IP)\rangle|
\le \E\sum_{k=1}^m \bigl|\E\bigl(F_k(u,u')\,|\,\FF_{k-1}\bigr)\bigr|
\le C_1\|u-u'\|^{\alpha\beta}.
$$
Taking the supremum over~$F$ with $\|F\|_\infty\le1$, we see that~\eqref{1.15} holds.

\subsection{Completion of the proof}
\label{s3.4}
We need to prove that Property~\hyperlink{(A)}{(A)} of Theorem~\ref{T:1.2} is satisfied for the Markov process~\eqref{1.1} and the mapping~$\varPsi$ constructed in Proposition~\ref{P:2.5} with an appropriate choice of~$\theta$. To this end, we fix $R>0$ so large that $X\subset B_H(R-1)$ and~$\KK\subset B_E(R)$, and choose~$\theta<1$ such that 
\begin{equation} \label{3.26}
\|S(u,\eta)-S(u',\eta)\|\le \theta^{-1}\|u-u'\|\quad
\mbox{for $u,u'\in B_H(R)$, $\eta\in B_E(R)$}. 
\end{equation}
Let us denote by $\delta>0$ and $\varPsi:X\times H\times E\to E$ the number and mapping constructed in Proposition~\ref{P:2.5}. Given $(u,u')\in X\times H$, let~$(u_k,v_k)$ be the random sequence given by~\eqref{1.8}, \eqref{1.9}. Without loss of generality, we assume that the underlying probability space~$(\Omega,\FF,\IP)$ coincides with the tensor product of countably many copies of~$(E,\BB(E),\ell)$ and denote by~$\{\FF_k\}_{k\ge1}$ the corresponding filtration. For any $(u,u')\in D_\delta$, let $N=N(u,u')\ge0$ be the smallest integer such that 
$$
\theta^{-N}\|u-u'\|\ge\delta.
$$
We define the sets~$\XXX_n$, $n\ge -N$ by the relation
\begin{equation} \label{3.27}
\XXX_n=\{(v,v')\in D_\delta:\theta^{n+1}\|u-u'\|<\|v-v'\|\le\theta^{n}\|u-u'\|\}.
\end{equation}
It is clear that the union of the sets $\cup_{n\ge -N}\XXX_n$ and the diagonal $\{(v,v):v\in X\}$ coincides with~$D_\delta$. Given $(u,u')\in D_\delta$, let us  consider a random sequence~$\{\xi_k\}_{k\ge0}$ given by\,\footnote{To simplify the notation, we do not indicate the dependence on $(u,u')$ for~$\XXX_n$ and~$\xi_k$ (as well as for the events~$\Gamma_k,\Gamma_k'$ and random variables $w_k, \zeta_k$ defined below). } 
$$
\xi_k=
\left\{
\begin{array}{cl}
+\infty&\mbox{if $u_k=v_k$},\\
n&\mbox{if $(u_k,v_k)\in\XXX_n$}, \\
-N-1 & \mbox{if $(u_k,v_k)\notin D_\delta$}. 
\end{array}
\right.
$$
In particular, we have $\xi_0=0$, and if $\xi_m=+\infty$ for some integer $m\ge1$, then $\xi_k=+\infty$ for $k\ge m$ (since $\varPsi^{u,u}(\eta)=\eta$ for any $u\in X$ and $\eta\in E$). Suppose we have proved that
 \begin{equation} \label{3.28}
\IP\{\xi_k\ge -l+ck\mbox{ for all $k\ge1$}\}\ge p_l 
\quad\mbox{for $\|u-u'\|\le \delta\theta^{2l}$},
\end{equation}
where  the sequence $\{p_l\}$ and the number $c>0$ do not depend on~$(u,u')$, and $p_l\to 1$ and $l\to \ty$. Then, in view of~\eqref{3.27}, on the set $\{\xi_k\ge -l+ck\}$, we have 
$$
\|u_k-v_k\|\le \theta^{-l+ck}\|u-u'\|
\le \delta^{1/2}\theta^{ck+l}\|u-u'\|^{1/2}\le  \theta^{ck}\|u-u'\|^{1/2},
$$
since we can assume that  $\delta<1$.
It follows that if we take  $q=\theta^c$  and~$\alpha=\frac12$, then the random time~$\tau$ defined by~\eqref{tau} will satisfy the inequality $\IP\{\tau=+\infty\}\ge p_l$. We thus obtain~\eqref{1.10}. Hence, it remains to prove~\eqref{3.28}. To this end, we shall use Corollary~\ref{c2.2}. 

\smallskip
If $\|u-u'\|\le \delta\theta^{2l}$ and $(u_{k-1}, v_{k-1})\in X_{n}$ for some integer $n\ge -2l$, then $\|u_{k-1}-v_{k-1}\|\le\delta$.  So inequality~\eqref{3.9} applies, and combining it with~\eqref{3.7} and~\eqref{3.26}, we see that
\begin{align}
\IP\{\xi_k-\xi_{k-1}\ge 1\,|\,\FF_{k-1}\}&\ge \frac34\quad
\mbox{on the set $\{\xi_{k-1}\ge -2l$\}},\label{3.29}\\
\IP\{\xi_k-\xi_{k-1}\ge -1\,|\,\FF_{k-1}\}&=1\quad
\mbox{almost surely},\label{3.30}
\end{align}
where $k\ge1$ is an arbitrary integer. Let us consider the event   
$$
\Gamma_k:=\{\xi_k-\xi_{k-1}\ge 1 \}.
$$
It follows from~\eqref{3.29} and~\eqref{3.30} that, with probability~$1$,
\begin{align*}
\E\{I_{\Gamma_k}|\,\FF_{k-1}\}&=\E\left\{I_{\Gamma_k} \left(I_{\{\xi_{k-1}\ge -2l\}} 
+I_{\{\xi_{k-1}< -2l\}}\right)|\,\FF_{k-1}\right\}\\
&\ge I_{\{\xi_{k-1}\ge -2l\}}\IP\{\xi_k-\xi_{k-1}\ge 1\,|\,\FF_{k-1}\}
\\&\quad +I_{\{\xi_{k-1}< -2l\}}\IP\{\xi_k-\xi_{k-1}\ge -1\,|\,\FF_{k-1}\}\ge \frac34.
\end{align*}
It is easy to see that the conditions of Lemma~\ref{l5.3} are satisfied with the following choice of the probability spaces and the function~$F$: the space $(\Omega_1,\FF_1,\IP_1)$ is the tensor product of $k-1$ copies of $(E,\BB(E),\ell)$, $(\Omega_2,\FF_2,\IP_2)$ coincides with $(E,\BB(E),\ell)$, and  $F:E\to\R$ is the orthogonal projection to the vector space of~$e_1$; see Hypothesis~\hyperlink{H4}{(H$_4$)}. Hence, there is a subset $\Gamma_k'\subset\Gamma_k$ such that $\E\{I_{\Gamma_k'}|\,\FF_{k-1}\}=\frac34$ almost surely. Define a random variable~$w_k$ by 
$$
w_k=\left\{
\begin{array}{cl}
1&\mbox{for $\omega\in\Gamma_k'$},\\[3pt]
-1&\mbox{for $\omega\in\Omega\setminus\Gamma_k'$}. 
\end{array}
\right. 
$$
The construction implies that~$w_k$ satisfies~\eqref{2.1}. Let us set $\zeta_k=w_1+\cdots+w_k$ and apply Corollary~\ref{c2.2} to find  a number $c>0$  and a sequence~$\{p_l\}$ converging to~$1$ as $l\to\infty$  such that 
\begin{equation} \label{3.31}
 \IP\{\zeta_k\ge -l+ck\mbox{ for all $k\ge1$}\} \ge p_l.
\end{equation}
  Now note that,   on the event in~\eqref{3.31}, we have $\xi_k\ge\zeta_k\ge -l+ck$, whence we conclude that~\eqref{3.28} is valid. This completes the proof of Theorem~\ref{t1.1}.

\section{Application}
\label{s4}

In this section, we apply Theorem~\ref{t1.1} to a parabolic PDE with a degenerate random perturbation. Namely, we consider Eq.~\eqref{rf-pde} in which $f:\R\to\R$ is  a polynomial of an odd degree $p\ge3$ with positive leading coefficient: 
\begin{equation} \label{nonlinear-term}
f(u)=\sum_{n=0}^p c_n u^n,	
\end{equation}
where $c_p>0$, and  $c_0, c_1, \ldots, c_{p-1}\in \R$ are arbitrary. In this case, it is easy to see that~$f$ satisfies the inequalities 
\begin{align}
	-C\le f'(u)&\le C(1+|u|)^{p-1},\label{E:4.1}\\
	f(u)u&\ge c\,|u|^{p+1}-C,\label{E:4.2}
\end{align}
where $u\in\R$ is arbitrary, and~$C,c>0$ are some constants. We shall confine ourselves to the case $p=5$ and $d=3$,  although all the results below remain true (with simple adaptations) in the case  
 \begin{equation}\label{power-dimension}
 \begin{cases}p\ge3 &   \text{ for }  d=1,2,\\ 3\le p\le\frac{d+2}{d-2} & \text{ for } d=3,4. \end{cases}	
 \end{equation}
 We    assume that $h\in H^1(\T^3)$ is a fixed function and~$\eta$ is a random process  of the form
\begin{equation} \label{E:4.3}
\eta(t,x)=\sum_{k=1}^\infty \I_{[k-1,k)}(t)\eta_k(t-k+1,x),
\end{equation}
where $\I_{[k-1,k)}$ is the indicator function of the interval $[k-1,k)$, and~$\eta_k$ are   i.i.d. random variables in $L^2(J,H)$ with $J:=[0,1]$ and $H:=L^2(\T^3)$.

Let us recall some definitions that were used in~\cite{KNS-2018} in the context of the Navier--Stokes system and complex Ginzburg--Landau equations. Given a finite-dimensional subspace~$\HH\subset H^2:=H^2(\T^d)$, we define by recurrence a non-decreasing sequence subspaces $\HH_k\subset H^2$ as follows: 
\begin{equation} \label{Hk}
\HH_0:=\HH, \quad \HH_{k+1}:=\text{span}\{\eta,\,\zeta\xi: \,\,\eta,\zeta\in\HH_k,\,\xi\in\HH\},\quad k\ge0.	
\end{equation}

\begin{definition}\label{D:4.1}
	A subspace $\HH\subset H^2$ is said to be   {\it saturating\/} if the union of~$\{\HH_k\}_{k\ge0}$ is dense in~$H$.
\end{definition}

 Examples of saturating spaces are provided by Proposition~\ref{P:5.2}.  Note that the saturation property does not depend on the number $\nu>0$ or on the polynomial~$f$. Let us denote by $(\cdot,\cdot)$  the scalar product in $H$.
 
 \begin{definition} \label{D:4.2}
	 A function $\zeta\in L^2(J,\HH)$  is said to be {\it observable\/}  if for any Lipschitz-continuous functions~$a_i:J\to\R$, $i\in\II$ and any continuous function $b:J\to\R$ the equality\,\footnote{It is easy to see that the observability of a function does not depend on the particular choice of the basis $\{\varphi_i\}$ in~$\HH$; see Remark~1.4 in~\cite{KNS-2018}.} 
$$
		\sum_{i\in\II}a_i(t)(\zeta(t),\varphi_i)-b(t)=0\quad\mbox{in $L^2(J)$}
$$
	implies that $a_i$, $i\in\II$ and~$b$ vanish identically. A probability measure~$\ell$ on~$L^2(J,\HH)$ is said to be    {\it observable\/} if $\ell$-almost every trajectory in $L^2(J,\HH)$ is observable.
\end{definition}

We now formulate the hypotheses imposed on the random process~\eqref{E:4.3}. We assume that it takes values in a finite-dimensional saturating subspace $\HH\subset H^2$. Let us fix an   orthonormal basis $\{\varphi_i\}_{i\in\II}$ in $\HH$, and denote by~$E_i$ the space~of square-integrable functions on~$J$ with range in~$\lspan(\varphi_i)$, so that $E:=L^2(J,\HH)$ is representable as the orthogonal sum of~$\{E_i\}_{i\in\II}$. We   assume that~$\ell=\DD(\eta_k)$ has a compact support $\KK\subset E$ containing the origin and satisfies the two hypotheses below.

\begin{description}
\item[Decomposability.] 
{\sl The measure~$\ell$ is representable as the tensor product of its projections~$\ell_i$ to~$E_i$. Moreover, the measures~$\ell_i$ are decomposable in the following sense: there is an orthonormal basis in~$E_i$ such that the measure~$\ell_i$ is representable as the tensor product of its projections to the one-dimensional subspaces spanned by the basis vectors. Finally, for any $i\in\II$ the corresponding one-dimensional projections of~$\ell_i$ possess Lipschitz-continuous densities with respect to the Lebesgue measure.} 
\item[Observability.]
{\sl
The measure~$\ell$ is observable.}  
\end{description}
 We refer the reader to Section~5 in~\cite{KNS-2018} for a discussion of  decomposability and observability properties and examples. In particular, it is shown there that both properties are satisfied for the Haar coloured noise given by~\eqref{eta-intro},~\eqref{0.3}.
  
  Let $(u_k, \pP_u)$ be the Markov process obtained by restricting the   solutions   of Eq.~\eqref{rf-pde} to integer times, and let $\PPPP_k$ and $\PPPP_k^*$ be the associated Markov semigroups. The following theorem is the main result of this section.
  
  \begin{theorem} \label{T:4.3}
	In addition to the above assumptions, suppose that the saturating  subspace $\HH$   contains the function identically equal to~$1$, and the dynamics of Eq.~\eqref{rf-pde} satisfies Hypotheses~\hyperlink{S}{\rm(S)} and \hyperlink{C}{\rm(C)} of the Introduction. Then, for any $\nu>0$, the   process~$(u_k,\IP_u)$ has a unique stationary measure~$\mu_\nu\in\PP(H)$,  
	and there is a sequence of positive numbers~$\{\gamma_k\}$ going to zero as $k\to\infty$ such~that
$$\|\PPPP_k^*\lambda-\mu\|_L^*\le \gamma_k\quad\mbox{for all $k\ge0$ and $\lambda\in\PP(H)$}. 
$$
\end{theorem} 

Before proving this theorem, let us consider a concrete example of a stochastic force for which the conclusion holds. To this end, we shall use some results described in the Appendix (see Sections~\ref{S:5.2}--\ref{S:5.4}). 

\begin{example} 
Let us denote by $\II\subset\Z^3$ the symmetric set defined in Proposition~\ref{P:5.2} and by~$\HH$ the corresponding $7$-dimensional subspace of trigonometric functions. We consider the process
\begin{equation*}
\eta^a(t,x)=a\sum_{l\in\II}b_l\eta^l(t)e_l(x), 	
\end{equation*}
where $a>0$ is a (large) parameter, $b_l\in\R$ are non-zero numbers, $\{e_l\}_{l\in\II}$ is the basis of~$\HH=\HH(\II)$ defined in Section~\ref{S:5.2}, and $\{\eta^l\}_{l\in\II}$ are independent Haar processes, see~\eqref{0.3}. Let us fix any $\nu>0$ and use Proposition~\ref{P:5.3} to find a subset $\GG_\nu\subset H^1(\T^3)$ of Baire's second category such that Hypothesis~\hyperlink{S}{(S)} is satisfied for any $h\in\GG_\nu$. We fix any~$h\in H^1(\T^3)$ with that property and denote by $w_1,\dots,w_N$ the corresponding set of solutions for~\eqref{stationary-eq}. As was explained in the Introduction, one of these solutions is locally asymptotically stable under the dynamics of the unperturbed equation~\eqref{unperturbed-eq}, and there is no loss of generality in assuming that~$w_N$ possesses that property. Let $\delta>0$ be a number such that the solutions of~\eqref{unperturbed-eq} issued from the $\delta$-neighbourhood of~$w_N$  satisfy~\eqref{wNconvergence}. In view of Theorem~\ref{T:5.5}, for any $i\in[\![1,N-1]\!]$, there is a smooth $\HH$-valued function $\zeta_i$ such that 
\begin{equation} \label{approxmate-control}
\|u(1;w_i,\zeta_i)-w_N\|<\delta,
\end{equation}
where $u(t;v,\eta)$ stands for the solution of~\eqref{rf-pde} corresponding to the initial state~$v\in L^2$ and the external force~$\eta$. Let $\KK^a\subset L^2(J,\HH)$ the support of the law~$\ell^a$ for the restriction of~$\eta^a$ to the interval $J=[0,1]$. Since the Haar functions~$\{h_0,h_{jl}\}$ entering~\eqref{0.3} form a basis in~$L^2(J)$, and the density~$\rho$ of the random variables~$\xi_k,\xi_{jl}$ is positive at zero, choosing~$a>0$ sufficiently large, we can approximate the functions~$\zeta_i$, within any accuracy  in~$L^2(J,\HH)$, by elements of~$\KK^a$. It follows that inequalities~\eqref{approxmate-control} remain valid for some suitable functions $\zeta_i\in\KK^a$, provided that $a\gg1$. Thus, Hypothesis~\hyperlink{C}{(C)} is also fulfilled. Finally, as is explained in Section~5 of~\cite{KNS-2018}, the measure~$\ell^a$ possesses the decomposability and observability properties. Hence, we can find $a_0(\nu,h)>0$ such that the conclusion of Theorem~\ref{T:4.3} is valid for any $\nu>0$, $h\in \GG_\nu$, and $a\ge a_0(\nu,h)$. 
\end{example}

\begin{proof}[Proof of Theorem~\ref{T:4.3}]
Let us denote by $S:H\times E\to H$, $u_0\mapsto u(1)$ the time-$1$ resolving operator for problem~\eqref{rf-pde},~\eqref{rf-pde0}. Due to the superlinear growth of~$f$ and parabolic regularisation property, there is a number $K>0$ such~that 
\begin{equation}\label{E:4.4}
\|S(u,\eta)\|_2\le K \quad \text{for any } u\in H, \eta\in \KK;
\end{equation}
see \cite[Lemma~2.10]{JNPS-cmp2014}. The theorem will be established if we check  Hypotheses~\hyperlink{H1}{(H$_1$)}--\hyperlink{H4}{(H$_4$)} of Theorem~\ref{t1.1} for $H=L^2$, $E=L^2(J,\HH)$, and $X=B_{H^2}(K)$. By construction, $X$~is compact in~$H$, and   inclusion $S(X\times\KK)\subset X$ follows from~\eqref{E:4.4}.  Hypothesis~\hyperlink{H1}{(H$_1$)} on the regularity of~$S$ is well known to hold for   Eq.~\eqref{rf-pde} (e.g., see Section~5 in~\cite[Chapter~1]{BV1992} and~\cite{kuksin-1982}), and  Hypothesis~\hyperlink{H4}{(H$_4$)} is satisfied  in view of the decomposability assumption. The remaining  hypotheses are checked in the following two~steps.
  
\smallskip
{\it Step~1. Checking Hypothesis~\hyperlink{H2}{\rm(H$_2$)}}. By Hypothesis~\hyperlink{S}{\rm(S)}, Eq.~\eqref{stationary-eq} has finitely many stationary states $w_1, \ldots, w_N$. As in the Introduction, $w_N$~is locally asymptotically stable and $\delta>0$ is its stability radius. We claim that Hypothesis~\hyperlink{H2}{(H$_2$)} is valid with $\hat u=w_N$. To see this, we first establish~\eqref{1.2} for $u\in W:=\{w_1,\ldots,w_{N-1}\}$ and an arbitrary~$\e>0$. Let us fix any $i\in[\![1,N-1]\!]$ and use Hypothesis~\hyperlink{C}{(C)} to find an integer $n_i\ge1$ and vectors $\zeta_{i1},\dots,\zeta_{in_i}\in\KK$ such that~\eqref{controlc} holds. Since the solutions of~\eqref{unperturbed-eq} that are issued from the $\delta$-neighbourhood of~$w_N$ converge uniformly to~$w_N$, we can find an integer $m\gg1$ such that~\eqref{1.2} holds for $u=w_i$ and $\hat u=w_N$, provided that $\zeta_j=\zeta_{ij}$ for $1\le j\le n_i$ and $\zeta_j=0$ for $n_i+1\le j\le m$. 

To check~\hyperlink{H2}{(H$_2$)} for arbitrary initial condition $u\in X$, we use the existence of a global Lyapunov function for the unperturbed equation~\eqref{unperturbed-eq}. Namely, let us set
\begin{equation}\label{E:4.5}
\Phi(u)=\int_{\T^3} \left(\frac\nu2|\nabla u|^2+F(u)-hu\right)\!\dd x,
\end{equation}
where $F(u)=\int_0^uf(s)\dd s$. Then, for any solution $u(t)$ of Eq.~\eqref{unperturbed-eq}, we have 
$$
\frac{\dd}{\dd t} \Phi(u(t))=\int_{\T^3} \p_tu\left(\nu \Delta u- f(u)+h \right)\dd x=-\int_{\T^3} \left(\p_tu\right)^2\dd x\le 0. 
$$
Thus, the function $t\mapsto \Phi(u(t))$ is non-increasing, and it is constant on a non-degenerate interval if and only if $u\equiv w_i$  for some $1\le i\le N$. Thus, $\varPhi$ is a global Lyapunov function for~\eqref{unperturbed-eq}. 

We now use a standard approach to prove that the $\omega$-limit set of any solution~$u(t)$ of Eq.~\eqref{unperturbed-eq} coincides with one of the stationary states (e.g., see Section~2 in~\cite[Chapter~3]{BV1992}). A simple compactness argument will then show  that the convergence to the stationary states is uniform with respect to the initial condition $u_0\in X$, and since $0\in \KK$, this will imply the validity of Hypothesis~\hyperlink{H2}{(H$_2$)}. 

To prove the required property, we first note that, for any $u_0\in X$, the trajectory $\{u(t), t\ge0\}$ is contained in the compact set~$X$, so that the corresponding $\omega$-limit set~$\omega(u_0)$ is non-empty. Since~$X$ is compact also in~$H^1$, for any $w\in \omega(u_0)$ we can find a sequence $t_n\to \ty$ such that   $u(t_n)\to w$ in~$H^1$ as $n\to\ty$. By the continuity of $\Phi: H^1\to \R$ and the monotonicity of $\Phi(u(t))$, we have
$$
\Phi(w)=\lim_{n\to \ty}\Phi(u(t_n))= \inf_{t\ge0} \Phi(u(t)). 
$$
On the other hand, the continuity of $S(\cdot,0):H^1\to H^1$ implies that 
$$
\Phi(S(w,0))=\lim_{n\to \ty}\Phi(S(u(t_n),0))=\lim_{n\to \ty}\Phi(u(t_n+1))=\inf_{t\ge0} \Phi(u(t)).
$$
This shows that $\Phi(w)=\Phi(S(w,0))$, so that $w$ is a stationary solution for~\eqref{unperturbed-eq}. Since~$\omega(u_0)$ is a connected subset, it must coincide with one of the stationary solutions. 

\smallskip
{\it Step~2. Checking Hypothesis~\hyperlink{H3}{\rm(H$_3$)}}. The verification of this hypothesis is similar to the cases of the Navier--Stokes system and complex Ginzburg--Landau equations considered in~\cite[Section~4]{KNS-2018}. Let us recall that the nonlinear term~$f:\R\to\R$ has the form~\eqref{nonlinear-term}, in which $p=5$, $c_5>0$, and $c_0,\dots,c_4\in\R$. It defines a smooth mapping in~$H^2$, whose derivative is a multiplication operator given by 
$$
f'(u;v)=f'(u)v=\biggl(\,\sum_{n=1}^5 n c_n u^{n-1}\biggr)v.
$$
We  need to show that the image of the derivative $(D_\eta S)(u,\eta):E\to H$ is dense  for any $u\in X$ and $\ell$-a.e.~$\eta\in E$. Let us fix $u\in X$ and $\eta\in E$, denote by $\tilde u\in L^2(J,H^3)\cap W^{1,2}(J,H^1)$ the solution of~\eqref{rf-pde}, \eqref{rf-pde0}, and consider the linearised problem
	\begin{equation} \label{E:4.6}
		\dot v-\nu \Delta v+f'(\tilde u(t))v=0, \quad v(s)=v_0, 
	\end{equation}
	where   $v_0\in H$. Let   $R^{\tilde u}(t,s):H\to H$ with $0\le s\le t\le 1$ be the resolving operator for this   problem. We define the Gramian  $G^{\tilde u}:H\to H$ by
	\begin{equation} \label{E:4.7}
	G^{\tilde u}:=\int_0^T R^{\tilde u}(T,t){\mathsf P}_\HH R^{\tilde u}(T,t)^*\dd t,
	\end{equation}
	where $R^{\tilde u}(T,t)^*:H\to H$ is the adjoint of $R^{\tilde u}(T,t)$, and~${\mathsf P}_\HH: H\to H$ is  the   projection to~$\HH$. Together with Eq.~\eqref{E:4.6}, let us consider its dual problem, which is a backward parabolic equation: 
	\begin{equation} \label{E:4.9}
	\dot w+\nu \Delta w-f'(\tilde u(t)) w=0, \quad w(1)=w_0.
	\end{equation}
This problem a unique solution $w\in L^1(J,H^1)\cap W^{1,2}(J,H^{-1})$  given by	\begin{equation} \label{E:4.10}
		w(t)=R^{\tilde u}(1,t)^*w_0. 
	\end{equation}	
	In view of Theorem 2.5 in~\cite[Part~IV]{zabczyk2008}, the image of~$(D_\eta S)(u,\eta)$ is dense in~$H$ if and only if
	\begin{equation}\label{E:4.8}
	\Ker(G^{\tilde u})=\{0\}.
		\end{equation}
We claim that this equality holds for any $u\in X$ and~$\ell$-a.e. $\eta\in E$. To prove this, we shall show that all the elements of    $ \Ker(G^{\tilde u})$ are  orthogonal to~$\HH_k$ for any~$k\ge0$. Since~$\cup_{k\ge0}\HH_k$ is   dense in~$H$, this will imply~\eqref{E:4.8}. 

We argue by induction on~$k \ge  0$. Let us take any $w_0\in \Ker(G^{\tilde u})$.  By~\eqref{E:4.7},  
$$
(G^{\tilde u}w_0,w_0)=\int_0^1\|{\mathsf P}_\HH R^{\tilde u}(1,t)^*w_0\|^2\dd t=0.
$$
This implies that   ${\mathsf P}_\HH R^{\tilde u}(1,t)^*w_0\equiv0$, and hence, for any $\zeta\in\HH_0$, we have
\begin{equation} \label{E:4.11}
	(\zeta,R^{\tilde u}(1,t)^*w_0)=0\quad \mbox{for $t\in J$}.
\end{equation}
Taking $t=1$, we see that $w_0$ is orthogonal to~$\HH_0$. Assuming that the function~$w_0$ is orthogonal to~$\HH_k$, let us prove its orthogonality to~$\HH_{k+1}$. We differentiate~\eqref{E:4.11} in time and use~\eqref{E:4.9} and \eqref{E:4.10} to derive 
$$
	\bigl(-\nu\Delta \zeta+f'(\tilde u(t))\zeta,w(t)\bigr)=0\quad\mbox{for $t\in J$}.
$$
Differentiating this equality in time and using~\eqref{E:4.9}, we obtain 
\begin{multline*}
\bigl(-\nu\Delta\zeta+f'(\tilde u)\zeta,\dot w\bigr)-\bigl(f^{(2)}(\tilde u;\zeta,-\nu\Delta\tilde u+f(\tilde u)-h),w\bigr)\\
+\sum_{i\in\II}\bigl(f^{(2)}(\tilde u;\zeta,\varphi_i),w\bigr)\eta^i(t)=0, 
\end{multline*}
where $\eta^i(t)=(\eta(t),\varphi_i)$ and  $f^{(k)}(u;\cdot)$ is  the $k^\text{th}$ derivative of~$f(u)$ (so that   $f^{(k)}=0$ for $k\ge6$). Setting
\begin{align*}
	a_i(t)&=\bigl(f^{(2)}(\tilde u;\zeta,\varphi_i),w\bigr),\\
	b(t)&=	\bigl(-\nu\Delta\zeta+f'(\tilde u)\zeta,\dot w\bigr)-\bigl(f^{(2)}(\tilde u;\zeta,-\nu\Delta\tilde u+f(\tilde u)-h),w(t)\bigr), 
\end{align*} we get  the equality
$$
b(t)+\sum_{i\in\II}a_i(t)\eta^i(t)=0\quad\mbox{for   $t\in J$},
$$ where $a_i$
 are Lipschitz-continuous functions and $b$ is continuous. The observability of~$\ell$ implies that  
\begin{equation*}
	\bigl(f^{(2)}(\tilde u(t);\zeta,\varphi_i),w(t)\bigr)=0
	\quad\mbox{for $i\in\II$, $t\in J$}. 
\end{equation*}
Applying exactly the same argument three more times, we derive   
$$
	\bigl(f^{(5)}(\zeta,\varphi_i,\varphi_j,\varphi_m,\varphi_n),w(t)\bigr)=0
	\quad\mbox{for $i,j,m,n\in\II$, $t\in J$}.
$$
 Taking $t=1$, we see that  $w(1)=w_0$ is orthogonal to the   space~$\VV$ spanned by~$\{(f^{(5)}(\zeta,\varphi_i,\varphi_j,\varphi_m,\varphi_n)\}$. As the space $\HH$   contains the function identically equal to~$1$, we can take $\varphi_j=\varphi_m=\varphi_n=1$, in which case   
$$
f^{(5)}(\zeta,\varphi,1,1,1)=120\,c_5  \zeta\varphi.
$$
The latter implies that~$\VV$ contains all the products $\zeta\xi$ with $\zeta\in\HH_k$ and $\xi\in\HH$. Combining this with the   induction hypothesis, we conclude that~$w_0$ is orthogonal to~$\HH_{k+1}$. This completes the proof of Theorem~\ref{T:4.3}.
\end{proof}

\section{Appendix}\label{s5} 
\subsection{Sufficient conditions for mixing}\label{s5.1}

Let $X$ be a compact metric space and let $(u_k,\IP_u)$ be a discrete-time Markov process 
in~$X$ possessing the Feller property. We denote by $P_k(u,\Gamma)$ the corresponding transition function, and by~$\PPPP_k$ and~$\PPPP_k^*$ the Markov operators. The following theorem is a straightforward consequence of Theorem~3.1.3 in~\cite{KS-book}. 

\begin{theorem} \label{t5.1}
Suppose that the following two conditions are satisfied for some point $\hat u\in X$. 
\begin{description}
\item[Recurrence.]
For any $r>0$, there is an integer $m\ge1$ and a number $p>0$ such that
\begin{equation} \label{5.1}
P_m\bigl(u,B_X(\hat u,r)\bigr)\ge p\quad\mbox{for any $u\in X$}.
\end{equation} 
\item[Stability.] 
There is a positive function $\delta(\e)$ going to zero as $\e\to0^+$ such that
\begin{equation} \label{5.2}
\sup_{k\ge0}\|P_k(u,\cdot)-P_k(u',\cdot)\|_L^*\le\delta(\e)
\quad\mbox{for any $u,u'\in B_X(\hat u,\e)$}.
\end{equation} 
\end{description}
Then the Markov process $(u_k,\IP_u)$ has a unique stationary measure $\mu\in\PP(X)$, 
and convergence~\eqref{1.6} holds.
\end{theorem}

To establish this theorem, it suffices to take two independent copies of the Markov process~$(u_k,\IP_u)$ and use standard techniques (based on the Borel--Cantelli lemma) to show that the first hitting time of any ball around $(\hat u,\hat u)$ is almost surely finite and has a finite exponential moment; combining this with the stability property, we obtain the required result. Since the corresponding argument is well known (e.g., see Section~3.3 in~\cite{KS-book}), we do note give more details.

\subsection{Saturating subspaces}
\label{S:5.2}
As in Section~\ref{s4}, we consider only the case~$d=3$; the other dimensions can be treated by similar arguments. For any non-zero vector $l=(l_1,l_2,l_3)\in \Z^3$, we set
 $$
c_l(x) =  \cos\lag l,x\rag, \quad 
s_l(x) =  \sin\lag l,x\rag, \quad x\in\T^3,
$$ 
where $\lag l,x\rag =  l_1x_1+l_2x_2+l_3x_3$. Let us define an orthogonal basis $\{e_l\}$ in~$L^2(\T^d)$ by the relation
$$
e_l(x)= 
\begin{cases} c_l(x)  & \text{if }l_1 > 0\text{ or } l_1=0,\, l_2 > 0 \text{ or } l_1=l_2=0,\, l_3 \ge 0, \\ s_l(x) & \text{if }l_1 < 0\text{ or } l_1=0,\, l_2 < 0\text{ or } l_1=l_2=0,\, l_3 < 0. 
\end{cases}
$$
Let $\II\subset\Z^3$ be a finite symmetric set (i.e., $-\II=\II$) containing the origin. We define 
\begin{equation}\label{E:5.3}
\HH(\II):=\lspan\{e_l:l\in\II\}
\end{equation}
and denote by~$\HH_k(\II)$ the sets~$\HH_k$ given by~\eqref{Hk} with $\HH=\HH(\II)$.  Recall that~$\II$ is called a {\it generator\/} if all the vectors in~$\Z^3$ are  finite linear combinations of elements of~$\II$ with integer coefficients. 
    
\begin{proposition}\label{P:5.2}
The subspace $\HH(\II)$ is saturating  if and only if~$\II$ is a generator. In particular, the set $\II=\{(0,0,0), (\pm1,0,0),(0,\pm1,0),(0,0,\pm1)\}$ gives rise to the $7$-dimensional saturating subspace~$\HH(\II)$. 
\end{proposition}

\begin{proof} 
To prove the sufficiency of the condition, we note that
\begin{equation}\label{E:5.4}
c_l(x)c_r(x)=\frac12\bigl(c_{l-r}(x)+c_{l+r}(x)\bigr), \quad 
s_l(x)s_r(x)=\frac12\bigl(c_{l-r}(x)-c_{l+r}(x)\bigr).
 \end{equation}
If $c_r, s_r\in\HH(\II)$ and~$c_l, s_l\in  \HH_k(\II)$, then~\eqref{E:5.4} implies that  $c_{l+r}, c_{l-r}\in \HH_{k+1}(\II)$.  A similar argument shows that~$s_{l+r},s_{l-r}\in  \HH_{k+1}(\II)$. Since~$\II$ is a generator, we see that all the vectors of the basis $\{e_l\}$ can be obtained from the elements of~$\HH(\II)$ after finitely many iterations. 
   
\smallskip
To prove the necessity, assume that~$\II$ is not a generator. Then there is a vector $m\in \Z^3$   that is not a finite linear combination of elements of $\II$ with integer coefficients.   It is easy to see that the functions $c_m$ and $s_m$ are orthogonal to $\cup_{k\ge0} \HH_k(\II)$. This shows that~$\HH(\II)$ is not saturating and completes the proof of the proposition.      
\end{proof}

\subsection{Genericity of Hypothesis~(S)}
\label{S:5.3}

\begin{proposition}\label{P:5.3}
Let $\nu>0$ be any number and let~$f$ be a real polynomial satisfying conditions~\eqref{E:4.1}--\eqref{power-dimension} with~$d=3$. Then there is a subset $\GG_\nu\subset H^1(\T^3)$ of Baire's second category such that, for any $h\in \GG_\nu$, the nonlinear  equation
\begin{equation} \label{stationary-eq3}
	-\nu\Delta w+f(w)=h(x), \quad x\in\T^3
\end{equation}
has finitely many solutions.
\end{proposition}

Before proceeding with the proof, let us recall the formulation of an infinite-dimensional version of Sard's theorem and some related definitions (see~\cite{smale-1965}). Let~$X$ and~$Y$ be Banach spaces. A linear operator $L:X\to Y$ is said to be {\it Fredholm\/} if its image   is closed, and the dimension of its kernel and the co-dimension of its image are finite. The {\it index\/} of~$L$ is defined by
$$
\Ind L:= \dim (\Ker L)- \codim(\Im L). 
$$
It is well known that if $L:X\to Y$ is a Fredholm operator and $K:X\to Y$ is a compact linear operator, then $L+K$ is also Fredholm, and $\Ind L=\Ind(L+K)$. A $C^1$-smooth map $F:X\to Y$ is said to be {\it Fredholm\/} if for any $w\in X$ the derivative $DF(w):X\to Y$ is a Fredholm operator. The {\it index\/} of~$F$ is the index of the operator~$DF(w)$ at some $w\in X$ (it is independent of the choice of $w$). A point~$y\in Y$ is called a {\it regular value\/} for~$F$ if $F^{-1}(y)=\varnothing$ or $DF(w):X\to Y$ is surjective for any~$w\in F^{-1}(y)$. The following result is due to Smale~\cite[Corollary~1.5]{smale-1965}. 

\begin{theorem}\label{T:5.4}
	Let $F:X\to Y$ be a $C^k$-smooth Fredholm  map such that $k>\max \{\Ind F,0\}$. Then its set of regular values is of Baire's second category. 
\end{theorem}

\begin{proof}[Proof of Proposition~\ref{P:5.3}] 
Let us consider the map 
$$
F:H^3(\T^3)\to H^1(\T^3), \quad  w\mapsto -\nu\Delta w+f(w).
$$
We have $\Ind(-\nu \Delta)=0$, so $\Ind(-\nu \Delta+Df(w))=0$ for any~$w\in H^3(\T^3)$, since the derivative $Df(w):H^3(\T^3)\to H^1(\T^3)$ (acting as the operator of multiplication by~$f'(w)$) is  compact. Smale's theorem implies the existence of a set $\GG_\nu\subset H^1(\T^3)$ of Baire's second category such that $DF(w):H^3(\T^3)\to H^1(\T^3)$ is surjective for any solution~$w$ of Eq.~\eqref{stationary-eq3} with~$h\in \GG_\nu$. Since the index is zero, it follows that the derivative $DF(w)$ is an isomorphism between the spaces~$H^3(\T^3)$ and~$H^1(\T^3)$ for any solution $w\in H^3(\T^3)$ of~\eqref{stationary-eq3}. Applying the inverse function theorem, we conclude that the solutions are isolated points in~$H^3(\T^3)$. On the other hand, the elliptic regularity implies that the family of all solutions for Eq.~\eqref{stationary-eq3} is a compact set in~$H^3(\T^3)$, so there can be only finitely many of them.
\end{proof}

\subsection{Approximate controllability of parabolic PDEs}\label{S:5.4}
In this section, we discuss briefly the  approximate controllability for Eq.~\eqref{rf-pde} established in~\cite{nersesyan-2018}.  This type of results were obtained by Agrachev and Sarychev~\cite{AS-2005, AS-2006} for the 2D  Navier--Stokes and Euler equations on the torus and later extended to the 3D case in~\cite{shirikyan-cmp2006,nersisyan-2010}. We assume that the  nonlinearity~$f:\R\to \R$ is a polynomial satisfying the hypotheses of Section~\ref{s4}.   The space  $\HH(\II)$ is   defined  by~\eqref{E:5.3} for some finite symmetric set~$\II\subset\Z^d$  containing the~origin, with an obvious modification of the functions~$e_l$ for $d=1,2, 4$.   

   \begin{theorem}\label{T:5.5} 
   In addition to the above   hypotheses, assume that~$\II$ is a generator for $\Z^d$ and~$h\in H^1(\T^d)$ is a given function. Then Eq.~\eqref{rf-pde} is approximately controllable in $L^2(\T^d)$, i.e.,     for any~$\nu>0$,    $\e>0$, and $u_0, u_1\in L^2(\T^d)$, there is a function   $\zeta\in L^2([0,1], \HH(\II))$   such that the solution of Eq.~\eqref{rf-pde} with initial condition~$u(0)=u_0$ satisfies the inequality 
$$
\|u(1) - u_1 \|_{L^2(\T^d)}<\e.
$$
\end{theorem}

This result is essentially Theorem~2.5 of~\cite{nersesyan-2018}, dealing with the case when the problem in question is not necessarily well posed and assuming that $u_0,u_1\in H^2(\T^3)$. Under our hypotheses, Eq.~\eqref{rf-pde} is well posed, and using a simple approximation argument, we can prove the validity of Theorem~\ref{T:5.5}. 

\addcontentsline{toc}{section}{Bibliography}
\def\cprime{$'$} \def\cprime{$'$}
  \def\polhk#1{\setbox0=\hbox{#1}{\ooalign{\hidewidth
  \lower1.5ex\hbox{`}\hidewidth\crcr\unhbox0}}}
  \def\polhk#1{\setbox0=\hbox{#1}{\ooalign{\hidewidth
  \lower1.5ex\hbox{`}\hidewidth\crcr\unhbox0}}}
  \def\polhk#1{\setbox0=\hbox{#1}{\ooalign{\hidewidth
  \lower1.5ex\hbox{`}\hidewidth\crcr\unhbox0}}} \def\cprime{$'$}
  \def\polhk#1{\setbox0=\hbox{#1}{\ooalign{\hidewidth
  \lower1.5ex\hbox{`}\hidewidth\crcr\unhbox0}}} \def\cprime{$'$}
  \def\cprime{$'$} \def\cprime{$'$} \def\cprime{$'$}
\providecommand{\bysame}{\leavevmode\hbox to3em{\hrulefill}\thinspace}
\providecommand{\MR}{\relax\ifhmode\unskip\space\fi MR }
\providecommand{\MRhref}[2]{%
  \href{http://www.ams.org/mathscinet-getitem?mr=#1}{#2}
}
\providecommand{\href}[2]{#2}

\end{document}